\documentclass[12pt]{amsart}

\usepackage{amsfonts,latexsym,rawfonts,amsmath,amssymb,amsthm,mathrsfs,esint,esvect}
\usepackage[plainpages=false]{hyperref}
\usepackage{graphicx}
\usepackage{comment}

\numberwithin{equation}{section}

\RequirePackage{color}
\theoremstyle{plain}

\newtheorem{lemma}{Lemma}[section]

\newtheorem{theorem}{Theorem}[section]
\newcommand{\beq}{\begin{equation}}
\newcommand{\eeq}{\end{equation}}
\newcommand{\beqs}{\begin{eqnarray*}}
\newcommand{\eeqs}{\end{eqnarray*}}
\newcommand{\beqn}{\begin{eqnarray}}
\newcommand{\eeqn}{\end{eqnarray}}
\newcommand{\beqa}{\begin{array}}
\newcommand{\eeqa}{\end{array}}

\def\S{\mathbb S}
\def\R{\mathbb R}
\def\M{\mathcal M}

\def\Om{\Omega}
\def\p{\partial}
\def\D{\nabla}

\def\s{\sigma_{\S^n}}

\def\dist{\text{dist}}

\def\ve{\varepsilon}

\def\lan{\langle}
\def\ran{\rangle}

\def\e{{\bf e}}

\begin{document}
\title{The Christoffel problem \\ by fundamental solution of the Laplace equation}

\author{Qi-Rui Li}
\address{Qi-Rui Li:
School of Mathematical Sciences, Zhejiang University, Hangzhou 310027, China;
and
Centre for Mathematics and Its Applications, the Australian National University, Canberra, ACT 2601, Australia.}
\email{qi-rui.li@zju.edu.cn}

\author{Dongrui Wan}
\address{Dongrui Wan: College of Mathematics and Statistics, Shenzhen University,  Shenzhen 518060, China}
\email{wandongrui@szu.edu.cn}

\author{Xu-Jia Wang}
\address{Xu-Jia Wang: 
Centre for Mathematics and Its Applications, the Australian National University, Canberra, ACT 2601, Australia.}
\email{xu-jia.wang@anu.edu.au}



\begin{abstract}
The Christoffel problem is equivalent to the existence of convex solutions 
to the Laplace equation on the unit sphere $\S^n$. 
Necessary and sufficient conditions have been found by Firey \cite{F67} and Berg \cite {B69},
using the Green function of the Laplacian on the sphere.
Expressing the Christoffel problem as the Laplace equation on the entire space $\R^{n+1}$,
we observe that the second derivatives of the solution 
can be given by the fundamental solutions of the Laplace equations.
Therefore we find new and simpler necessary and sufficient conditions for the solvability of the Christoffel problem.
We also study the $L_p$ extension of the Christoffel problem and provide sufficient conditions
for  the problem,  for the case $p\ge 2$.
\end{abstract}

\maketitle

%%%%%%%%%%%%%%%%%%%%%%%%%%%%%%%%%%%%%%%%%%%%%%%%%%%%%%%%
\baselineskip16.4pt
\parskip3pt

%%%%%%%%%%%%%%%%%%%%%%%%%%%%%%%%%%%%%%%%%%%%%%%%%%%%%%%

\section{Introduction}

Given a positive function $f$ on the unit sphere $\S^n$,
the Christoffel problem concerns the existence of a closed convex body $\Om\subset \R^{n+1}$ 
such that the sum of the principal curvature radii of $\M = \p\Om$ at $p$ is equal to $f(x)$,
where $x$ is the unit outward normal of $\M$ at $p$.
This problem has been studied by many authors \cite{C65, Bu, Hur02, Hil12, Su33, Alex37, Pog53}. 
Firey \cite{F67} and Berg \cite {B69} finally obtained
necessary and sufficient conditions for the solvability of the problem.
We refer the readers to \cite{F67, GYY, GM03} for more details on early works on the Christoffel problem.
In particular a nice sufficient condition was found in \cite{GM03}.

Let $u(x)=\sup\{x\cdot p\ |\ p\in\Om\}$, where $x\in \S^n$,
be the support function of $\Om$. 
It is well known that the eigenvalues of the Hessian matrix $U=:\{\D^2 u(x)+u(x)I\}$
are the principal radii of $\M$ at $p$,
where $\D$ denotes the derivative with respect to an orthonormal frame on $\S^n$.
Hence Christoffel's problem is equivalent to the existence of convex solutions  to  
\beq\label{e1}
\Delta_{\S^n} u + n u  = f \ \ \text{on} \ \S^n,
\eeq
where $\Delta_{\S^n}$ is the Laplacian on $\S^n$. 
A necessary and sufficient condition for the existence of solutions to \eqref{e1}  is that 
 $f$ is orthogonal to the kernel of the operator $\Delta+nI$ on $\S^n$, namely
\beq\label{OT}
 \int_{\S^n} x_i f(x) dx = 0 , \ \ i=1,\dots, n+1.
\eeq

The key point to solve Christoffel's problem is to find conditions such that the solution is convex, 
namely the matrix  $U\ge0$. 
Firey \cite{F67} derived the Green function $G_{\S^n}(x,y)$ of the operator  $\Delta_{\S^n} + n I $ on $\S^n$.
Hence the solution to \eqref{e1} is given by $u(x)=\int_{\S^n} G_{\S^n}(x, y) f(y) dy$, 
and the matrix  $U\ge 0$ is equivalent to  
\beq \label{cr}
\Big\{  \int_{\S^n} \p_{x_i} G_{\S^n}(x, y) \p_{y_j} f (y) dy \Big\} \ge 0 . 
\eeq
Unfortunately the Green function on $\S^n$ is not explicitly given,
and the condition \eqref{cr} is not easy to verify.
We remark that $U\equiv 0$ if and only if $u$ is linear, i.e., $\Om$ is a point. 
But we assume that $f>0$, so this case does not occur.

Berg \cite{B69} studied Christoffel's problem after Firey. 
He deduced a recursion relation on the dimension for the expression of solutions.
In \cite{GYY} the authors re-created  Berg's recursion relation by a different method. 
But again the recursion formulas in  \cite{F67, GYY} are rather complicated,
and not easy to verify.
We will state the conditions of Firey \cite {F67} and Berg \cite{B69} in Remark 3.1 
for reader's convenience.

In this paper we observe that the second derivatives of the solution can be expressed by the
fundamental solution of the Laplacian operator on $\R^{n+1}$,
instead of the Green function on $\S^n$.
Therefore we found much simpler conditions for the convexity of solutions.  
More precisely, 
we extend $u$ to $\R^{n+1}$ such that it is homogeneous of degree $1$, 
and extend $f$ to $\R^{n+1}$ such that it is homogeneous of degree $-1$.
Then equation \eqref{e1} is equivalent to  
\beq\label{e2}
\Delta u  = f \ \ \text{on} \ \R^{n+1}.
\eeq
Note that $u$ and $f$ have a singular point at $0$, so that $u$ is a weak solution in $W^{2,p}_{loc}(\R^{n+1})$ for $p\in (1, n)$.
It is amazing to see that Christoffel's problem is equivalent to the existence of convex solutions to
such a simple equation   \eqref{e2}.
Note that if $u$ is a convex solution to \eqref{e2}, then $\M$ can be recovered from $u$ by
$\M=\{Du(x)\ |\ x\in \S^n\}$.

If $n=1$, $\M$ is a curve in $\R^2$.
In this case, $U(x)=f(x)$
and  the solution is convex if and only if $f(x)\ge 0$ but  $f\not\equiv 0$. 
In the following we focus on the case $n\ge2$.
Denote
\beq\label{wr}
w_{R}(x) = \int_{B_R(0)} F(x,y) f(y) d y, 
\eeq
where $F(x,y)$ is the   fundamental solution of the Laplace equation in $\R^{n+1}$,
\beq\label{fs}
F(x,y) = \frac{1}{(1-n)\omega_n}|x-y|^{1-n},
\eeq
and $\omega_n=|\S^n|$.
For any fixed $x$, when $R$ is sufficiently large, we have $F(x, y) f(y) = O(|y|^{-n})$. 
Hence $w_R =O(R)$ in $B_R(0)$.  
Noting that $u$ is homogeneous of degree 1, the harmonic function $h_R =: u-w_R =O(R)$ in $B_R(0)$. 
Hence we have $D h_R=O(1)$ and $D^2 h_R=O(R^{-1})$  in $B_R(0)$.
Therefore if the limits $w_{ij} (x)=:\lim_{R\to\infty} \frac{\p^2 w_R}{\p x_i \p x_j} (x)$ exist, we have
\beq\label{uw}
  u_{x_ix_j}(x)=w_{ij}(x)\ \ \forall\ x\in\R^{n+1}.
\eeq
Note that on the right hand side,  $i, j$ are understood as subscripts, not derivatives.

We compute the first derivative of $w_R$,
\beq\label{1st}
 \begin{split}
D_iw_R(x) & =  \int_{B_R(0)}  f(y) F_{x_i}(x, y)dy =-  \int_{B_R(0)}  f(y)F_{y_i}(x, y) dy \\
             & = - \int_{\p B_R(0)} f(y) F(x, y) \gamma_i  + \int_{B_R(0)}  f_{y_i} F(x, y)dy,
 \end{split}
\eeq
and the second derivatives
\beq\label{2nd}
D_{ij}w_R(x) = - \int_{\p B_R(0)} f(y) F_{x_j}(x, y) \gamma_i + \int_{B_R(0)}  f_{y_i} F_{x_j}(x, y)dy ,
\eeq
where $\gamma$ is the unit normal to $\p B_R(0)$.
The first integral in \eqref{2nd}            
$$  \int_{\p B_R}f(y)  F_{x_j}(x, y) \gamma_i 
=\frac{1}{\omega_n} \int_{\p B_R}   \frac {f(y) (x_j-y_j)\gamma_i}{|x-y|^{n+1}} =O(R^{-1}) \to 0.$$                     
Hence               
\beq\label{wij}
{\begin{split} 
\lim_{R\to\infty} D_{ij}w_R(x)  
                       & =  \int_{\R^{n+1}}  f_{y_i} F_{x_j}(x, y) dy\\
                       & = \frac{1}{\omega_n} \int_{\R^{n+1}}  \frac {(x_j-y_j)f_{y_i}}{|x-y|^{n+1}} =: w_{ij} (x) .
\end{split}}
\eeq
For any point $x\in \S^n$, in order that $u$ is convex at $x$, by the homogeneity of $u$,
it suffices to verify that for any unit vector $\xi=(\xi_1,\cdots,\xi_{n+1})$ satisfying $\xi\perp x$,
there holds
$\sum_{i, j=1}^{n+1} w_{ij}\xi_i\xi_j \ge 0$.
Therefore we obtain the following criterion for the convexity of solutions to \eqref{e2}.

\begin{theorem}\label{T1.1} 
Suppose $n\ge 2$ and $f$ is a positive and Lipschitz continuous function.
The solution $u$ is convex
if and only if \ $\forall\ x\in \S^n$ and $\forall$ unit vector $\xi\perp x$, there holds
\beq\label{cr1}
\int_{\R^{n+1}}  \frac {-\lan \xi, y\ran f_\xi}{|x-y|^{n+1}}dy \ge 0.
\eeq
\end{theorem}

Our condition \eqref{cr1} looks much simpler than \eqref{cr},
because the Green function on the sphere is very complicated.
But we should point out that our condition is equivalent to Firey's, 
and also Berg's conditions, as they are all necessary and sufficient conditions
for the Christoffel problem, see Remark 3.2. 
In Theorem \ref{T3.1}-\ref{T3.3} below, we will give sufficient conditions for H\"older continuous $f$.
Note when $n=1$, $\M$ is a closed curve and is automatically convex when $f> 0$.

We point out that condition \eqref{cr1} is only used for the convexity of solution. 
For the existence of solutions to \eqref{e1}, 
one also needs to assume the condition \eqref{OT} as in \cite{F67}.
Combining the existence of solutions in \cite{F67} and the convexity of solutions (Theorem \ref{T1.1}), 
we have

\begin{theorem}\label{T1.2} 
Suppose $n\ge 2$ and $f$ is a positive Lipschitz continuous function.
The Christoffel problem has a solution if and only if $f$ satisfies \eqref{OT} and \eqref{cr1}.
Moreover, the solution is unique up to translation. 
\end{theorem}

This paper is divided into two parts.
In part I, we will discuss in more details on the condition \eqref{cr1}.
For example, by integration by parts, we see that the solution is convex
if $f$ is H\"older continuous and its H\"older norm is smaller than a certain constant. 
The constant is computed in \eqref{gna}.
We will also deduce other sufficient conditions on $f$ such that the solution is convex.
These conditions are  contained in Theorems  \ref{T3.1}-\ref{T3.3}.
In Section 2 we also give a proof for the existence of entire solutions to equation \eqref{e2}, 
for any locally integrable function $f$.

In part I\!I, we consider an extension of the Christoffel problem, called the $L_p$-Christoffel problem.
The associated equation is 
\beq\label{e3}
\Delta_{S^n} u +nu  = f(x) u^{p-1} \ \ \text{on} \  \S^{n}.
\eeq
This equation was first introduced in \cite{HMS04} and later studied in \cite{GX18}.
In these two papers, the authors extended the sufficient condition for the Christoffel-Minkowski problem in \cite {GM03}
to equation \eqref{e3}, for $p>1$.
In this paper we give some sufficient conditions for the convexity of solutions to \eqref{e3} in 
Theorem \ref{T4.1}.  
 \vskip20pt

\section{Existence of entire solutions to the Laplace equation}

In this section, we show that equation \eqref{e2} has a solution for any locally integrable $f$. 
This is a known result but it is hard to find a proof in literature. So we present a proof here
which should be of interest to the readers.

\begin{theorem}\label{T2.1}
Suppose $f(x)\in L^p_{loc}(\R^n)$ for some $p\ge 1$.
Then there exists a solution $u$ to the equation
\beq\label{2.1}
\Delta u=f(x)\ \ \ \text{on}\ \ \R^n.
\eeq
\end{theorem}

 \begin{proof}
At first we consider the 2-dimensional case and
suppose $f(x)$ is a locally bounded function.
Let
$$ {\begin{split}
G(x, y) = : & \frac {1}{2\pi}\big(\log |x-y|-\log |y|\big)\\
 =\ & \frac {1}{4\pi}\log \Big(1+\frac{|x|^2-2x\cdot y}{|y|^2}\Big)\\
=&-\frac {1}{4\pi}\sum_{j=1}^\infty\frac 1j  \Big(\frac{2x\cdot y-|x|^2}{|y|^2}\Big)^j\\
=& \sum_{|\alpha|\ge 1}\phi_\alpha(y)x^\alpha,
\end{split}} $$
where  $\alpha=(\alpha_1,\ldots, \alpha_n)$ with $\alpha_i\ge0$, 
$x^\alpha =\prod_{i=1}^n x_i^{\alpha_i}$,
and $\phi_\alpha(y)$ satisfies
$$\phi_\alpha(ry)=\frac 1{r^{|\alpha|}}\phi_\alpha(y).$$
Since $\Delta_xG(x, y)=0$, we have
$$\Delta_x \Big(\sum_{|\alpha|=k}\phi_\alpha(y) x^\alpha\Big)=0\ \ \forall\ \ k\ge 1.$$
Moreover,
$$|\phi_\alpha(y)|\le M^{|\alpha|}/|y|^{|\alpha|}$$
for some $M>0$ independent of $|\alpha|$.
% Indeed, $|\phi_\alpha(y)|=|y|^{-\alpha}\cdot |\phi_\alpha(\frac y{|y|})|
%\le M^{|\alpha|}/|y|^\alpha$.

Let
$$\psi_k(x, y)=\sum_{|\alpha|\le k}\phi_\alpha(y) x^\alpha,\ \ k\ge 1.$$
We can select a monotone increasing sequence
$R_k\to \infty$ as $k\to \infty$ so that
$$\sum_{|\alpha|=k}^\infty |\phi_\alpha (y)||x|^{|\alpha|}
\le \frac{1}{1+|y|^{n+2}} \frac{1}{\sup_{|x|\le R_k} (1+|f(x)|)}
\ \ \forall\ 1< |y|\le R_k  \ \text{and}\ |x|\le |y|/4M.$$
This is possible since the left hand side
$$\le \sum_{|\alpha|=k}^\infty
\frac {M^{|\alpha|}|x|^{|\alpha|}}{|y|^{|\alpha|}}
\le \sum_{|\alpha|=k}^\infty \frac {1}{4^{|\alpha|}}
\le \sum_{j=k}^\infty 2^{-j}$$
provided $|y|\ge 4M|x|$.

Let $\Om_k=B_{R_k}(0)\backslash B_{R_{k-1}}(0)$. Let
$$u_k(x)=\sum_{j=1}^k\int_{{\Om_j}} f(y) \big[G(x,y)-\psi_j(x,y)\big]dy.$$
Then
$$\Delta u_k(x)=
\left\{ {\begin{split} 
 & f(x),\ \ \ \ \          &x\in B_{R_k}(0),\\
  & 0 ,  \ \ \ \ \  \ \            &|x|> R_k .
\end{split}} \right.$$
For $|x|\le R_k/4M$, we estimate $u_k(x)$ as follows.
$${\begin{split} 
|u_k(x)|
&\le \sum_{j=1}^{j_0} \Big| \int_{\Om_j} f(y)\big[G(x,y)-\psi_j(x,y)\big]dy \Big|\\
&+\sum_{j=j_0+1}^{k} \Big|\int_{\Om_j} f(y)\big[G(x,y)-\psi_j(x,y)\big]dy\Big|
\end{split}}  $$
where $j_0$ is {the smallest integer}
so that $R_{j_0}\ge 4M|x|$ ($j_0$ depends on $|x|$ and $f$,  but is independent of $k$).
The first term
$${\begin{split} 
&\le \sum_{j=1}^{j_0} \Big| \int_{\Om_j} f(y)[G(x,y)-\psi_j(x,y)]dy \Big|\\
&\le C_{f, x} \int_{B_{R_{j_0}}(0)}|G(x,y)|+\sum_{j=1}^{j_0}|\psi_{j}(x,y)|dy\\
&\le \tilde C_{f, x} .
\end{split}}  $$
The second term
$${\begin{split} 
&\le \sum_{j=j_0+1}^{k} \Big| \int_{\Om_j} f(y) \frac{1}{1+|y|^{n+2}} \frac 1{\sup_{|x|\le R_j} (1+|f(x)|)}dy \Big| \\
&\le C \int_{B_{R_k}(0)\backslash B_{R_{j_0}}(0)}\frac{1}{1+|y|^{n+2}} dy\\
&\le C .
\end{split}}  $$
Hence for any given $|x|\le R_k/4M$,
$$u_k(x)\le C$$
with $C$ depending only on $x$ and $f$, but is independent of $k$.
Passing to a subsequence we see that
$u_k(x)\to u(x)$ as $k\to\infty$ and $u$ is a solution of \eqref{2.1}.
Moreover, $u(x)$ can be represented by
$$u(x)=\sum_{j=1}^\infty \int_{\Om_j} f(y)\big[G(x,y)-\psi_j(x,y)\big]dy.$$

For general $f(x)\in L_{loc}^p(\R^n)$, we may decompose $f(x)=g(x)+h(x)$ so
that $g(x)$ is a locally bounded function and $h(x)\in L^p(\R^n)$.
From the above proof there exists a solution to $\Delta v(x)= g(x)$.
The existence of $\Delta w(x)=h(x)$ is known, see, e.g. \cite{AGG}.
Hence $v(x)+w(x)$ is a solution of \eqref{2.1}.

For higher dimensional case, let
$${\begin{split} 
G(x, y) &=\frac {-1}{(n-2)\omega_{n-1}}\big\{|x-y|^{2-n}-|y|^{2-n}\big\} \\
     &= \frac {-1}{(n-2)\omega_{n-1}|y|^{n-2}}
\Big\{ \frac {1}{\big(1+\frac {|x|^2-2x\cdot y}{|y|^2}\big)^{\frac{n-2}{2}}} -1\Big\} \\
  &= \sum_{|\alpha|\ge 1}\phi_\alpha(y)x^\alpha  
  \end{split}}  $$
where $\phi_\alpha(y)$ satisfies
$$\phi_\alpha(ry)=r^{2-n-|\alpha|}\phi_\alpha(y),$$
and
$$\Delta_x(\sum_{|\alpha|=k}\phi_\alpha(y) x^\alpha)=0\ \ \forall\ \ k\ge 1.$$
Moreover,
$$|\phi_\alpha(y)|\le M^{|\alpha|}/|y|^{|\alpha|+n-2}$$
for some $M>0$ independent of $|\alpha|$.
Suppose $f(x)\in L^\infty_{loc}(\R^n)$. Let
$$\psi_k(x, y)=\sum_{|\alpha|\le k}\phi_\alpha(y) x^\alpha,\ \ k\ge 1 ,$$
and select an increasing sequence $R_k\to \infty$ as $k\to \infty$ so that,
for $k$ large,
$$\sum_{|\alpha|=k}^\infty |\phi_\alpha (y)||x|^{|\alpha|}
\le \frac{1}{1+|y|^{n+2}} \frac{1}{\sup_{|x|\le R_k} (1+{|f(x)|})}
\ \ \forall\  {1< |y|\le R_k} \ \text{and}\ |x|\le |y|/4M .$$
Let $\Om_k=B_{R_k}(0)\backslash B_{R_{k-1}}(0)$ and let
$$u_k(x)=\sum_{j=1}^k\int_{\Om_j} f(y)[G(x,y)-\psi_j(x,y)]dy.$$
Similarly to the case $n=2$ ,we have $|u_k(x)|\le C$ for $|x|\le R_k/4M$,
where $C$  is independent of $k$. Hence by selecting a subsequence we see that
$u_k$ converges to a solution of \eqref{2.1}.
\end{proof}

We point out that equation \eqref{e2} has a solution even if $f$ is not homogeneous of degree $-1$.
But in this case, the solution is not homogeneous of degree 1.

%%%%%%%%%%%%%%%%%%%%%%%%%%%%%%%%%%%%%%%%%%%%%%%%%%%%%%%

\section{Convex solutions to Christoffel's problem}

In this section, we deduce more conditions on $f$
such that the solution to \eqref{e2} is convex.

First  by \eqref{2nd}, we have
\beqn\label{t3.1}
D_{ij}w_R(x) &=& -\int_{\p B_R(0)}f(y) F_{x_j}(x,y)\gamma_i
+f(x)\int_{B_R(0)}  F_{x_j}(x,y) (|y|^{-1} )_{y_i}dy\notag\\
&&+ \int_{B_R(0)}  F_{x_j}(x,y)\big(f(y) - |y|^{-1} f(x)\big)_{y_i} dy \notag \\
&=& -f(x)\int_{\p B_R(0)}|y|^{-1}F_{x_j}(x,y)\gamma_i
+f(x)\int_{B_R(0)}  F_{x_j}(x,y) (|y|^{-1} )_{y_i}dy\\
&&+ \int_{B_R(0)}  F_{x_ix_j}(x,y)\big(f(y) - |y|^{-1} f(x)\big)dy .\notag
\eeqn
The first integral in \eqref{t3.1}
\beqs
\int_{\p B_R(0)}|y|^{-1}F_{x_j}(x,y)\gamma_i =\frac{1}{\omega_nR} 
\int_{\p B_R(0)}\frac{x_j-y_j}{|x-y|^{1+n}}\gamma_i =O(R^{-1}).
\eeqs
Noting that $u_0(x)=: \frac1n |x|$ solves \eqref{e2} for $f = |x|^{-1}$,
we can calculate the second integral in \eqref{t3.1},
\beqs
\lim_{R\to\infty}\int_{B_R(0)} F_{x_j}(x,y) (|y|^{-1})_{y_i}dy = 
\frac1n \Big(\frac{\delta_{ij}}{|x|} - \frac{x_ix_j}{|x|^3}\Big).
\eeqs
Sending $R\to \infty$, we conclude by \eqref{uw}
\beq\label{t3.2}
u_{ij}(x) = \frac{f(x)}{n}\Big(\frac{\delta_{ij}}{|x|} - \frac{x_ix_j}{|x|^3}\Big)
+\int_{\R^{n+1}} F_{x_ix_j}(x,y)\big(f(y) - |y|^{-1} f(x)\big)dy.
\eeq
Note that the integral \eqref{t3.2} is convergent provided $f\in C^\alpha(\S^n)$ for some $\alpha>0$.
By the homogeneity,  $u$ is convex if and only if $\sum_{i,j}u_{ij}(x)\xi_i\xi_j \ge 0$
for all unit vector $\xi$ satisfying $\xi\perp x$.
Direct computation shows
\beqs
F_{x_ix_j}(x,y) = \frac{1}{\omega_n} \frac{|x-y|^2\delta_{ij} -(n+1)(x_i-y_i)(x_j-y_j)}{|x-y|^{n+3}}.
\eeqs
Therefore we obtain the following criterion for the convexity of $u$.

\begin{theorem}\label{T3.1}
Let $n\ge2$, and $f\in C^\alpha(\S^n)$ be a positive function.
Extend $f$ to $\R^{n+1}$ such that it is homogeneous of degree $-1$. 
Then the solution $u$ to \eqref{e1} is convex
if and only if \ $\forall \ x,\xi \in\S^n$ with $\xi\perp x$, 
\beq\label{cr2}
\frac{1}{\omega_n}\int_{\R^{n+1}}\frac{|x-y|^2-(n+1)\lan\xi,y\ran^2}{|x-y|^{n+3}} \Big(f(y)-f(x)/|y|\Big)dy
+ \frac1nf(x) \ge 0.
\eeq
\end{theorem}

From \eqref{cr2}, we can give a quantitative condition on $f$  
such that \eqref{e1} admits a convex solution.
For $\alpha>0$, denote 
\beq\label{gna}
\gamma_{n,\alpha}=\frac{\omega_n}{n(n+1)}\Big[\int_{\R^{n+1}}|y-\e_{n+1}|^{-n-1}|y|^{-1}
\dist_{\S^n}^\alpha(y/|y|,\e_{n+1})dy\Big]^{-1},
\eeq
where $\dist_{\S^n}(x,z)$ is the spherical distance between two points $x,z\in\S^n$,
and $\e_{n+1}=(0,\ldots,0,1)$.
Note that $\gamma_{n,\alpha}$ is unchanged if $\e_{n+1}$ is replaced by any other $x\in\S^n$.

\begin{theorem}\label{T3.2}
Let $n\ge 2$, and $f\in C^\alpha(\S^n)$ be a positive function, $\alpha\in (0,1)$.
Then the solution to \eqref{e1} is convex provided
\beq\label{fca}
|f|_{C^\alpha(\S^n)} \le \gamma_{n,\alpha} \min_{\S^n}f . 
\eeq
\end{theorem}

\begin{proof}
For $x,\xi\in\S^n$ with $\xi\perp x$ and $y\in\R^{n+1}$,
we have 
$$\lan \xi,y \ran^2 = \lan \xi,y-x\ran^2\le|x-y|^2 , $$
with equality if and only if $y-x = t\xi$ for some $t\in\R$.
Therefore
\beqs
\text{LHS of \eqref{cr2}}&>&
-\frac{(n+1)|f|_{C^\alpha(\S^n)}}{\omega_n} \int_{\R^{n+1}}|x-y|^{-n-1} |y|^{-1} \dist_{\S^n}^\alpha(y/|y|,x)dy
+\frac1n\min_{\S^n}f\\
&\ge&0,
\eeqs
provided \eqref{fca} holds.
\end{proof}

By rewriting \eqref{cr1} and \eqref{cr2} in their equivalent form, we have the following theorem.

\begin{theorem}\label{T3.1a}
Let $n\ge 2$,  $f>0$ be a function on $\S^n$.
Extend $f$ to $\R^{n+1}$ such that it is homogeneous of degree $-1$. 
\begin{itemize}
\item[(i)] If $f\in C^{0,1}(\S^n)$, then the solution $u$ to \eqref{e1} is convex if and only if 
\beq\label{cr3a}
\int_{\S^n}\omega(x,z)\lan\xi,z\ran\lan Df(z),\xi \ran dz\ge 0 \ \ \forall \ x, \xi\in\S^n  \ \text{satisfying} \ x\perp \xi,
\eeq
where
\beqs
\omega(x,z) = - \int_0^\infty \frac{r^{n-1}}{|x-rz|^{n+1}}dr .
\eeqs

\item[(ii)] If $f\in C^{\alpha}(\S^n)$ for some $\alpha\in (0, 1)$,
then the solution $u$  is convex if and only if 
\beq\label{cr3b}
\int_{\S^n} \hat \omega(x,z,\xi) g(x,z)d\s(z) + \frac1n f(x) \ge 0
\ \ \forall \ x, \xi\in\S^n  \ \text{satisfying} \ x\perp \xi,
\eeq
where
\beqs
\hat \omega(x,\xi,z)&=& \frac{1}{\omega_n}\int_0^\infty
\frac{|x-rz|^2-(n+1)\lan\xi,rz\ran^2}{ |x-rz|^{n+3}}r^{n-1}dr, \\
g(x,z)&=&f(z)-f(x).
\eeqs

\end{itemize}
\end{theorem}

\begin{proof}

We first show that \eqref{cr3a} is necessary and sufficient for the convexity.
Direct calculation shows that, by the homogeneity of $f$,
\beqn\label{T3.1a.1}
\text{LHS of \eqref{cr1}} &=&- \int_0^\infty\int_{\S^n} \frac{r^{n-1}}{|x-rz|^{n+1}}
 \lan \xi,z\ran \lan Df(z),\xi\ran drdz \notag \\
&=&  \int_{\S^n}\omega(x,z) \lan \xi,z\ran \lan Df(z),\xi\ran dz.
\eeqn
Note that 
$$\lan \xi,z\ran = \lan \xi,z-x\ran = O(|x-z|) . $$
We will verify that
\beq\label{oxz}
 \omega(x,z) = O(|x-z|^{-n} )\ \ \ \text{for  $z$ near $x$.}
\eeq
Hence the integral \eqref{T3.1a.1} is convergent,
and by Theorem \ref{T1.1} we finish the proof of part (i).
To verify \eqref{oxz}.
we assume by a rotation of coordinates that
$x=\e_{n+1}$ and $z = (\ve\e_1+\e_{n+1})/\sqrt{1+\ve^2}$,
where $\e_k$ is the unit vector on the $x_k$-axis.
Then
\beqs
\omega(x,z)&\simeq&-\int_{0}^2 \frac{r^{n-1}}{\big((1-r)^2+\ve^2r^2\big)^{\frac{n+1}{2}}}dr
\simeq -\int_{-1}^1 \frac{(1+t)^{n-1}}{\big(t^2+\ve^2\big)^{\frac{n+1}{2}}}dt\\
&\simeq& -\ve^{-n}\int_{-1/\ve}^{1/\ve} \frac{1}{\big(\rho^2+1\big)^{\frac{n+1}{2}}}d\rho \\
&=& O(\ve^{-n}).
\eeqs
Hence \eqref{oxz} follows.

We next prove part (ii). By computation,
\beqn\label{T3.1a.3}
\text{LHS of \eqref{cr2}} &=& \frac{1}{\omega_n} \int_{\S^n} \int_0^\infty
\frac{|x-rz|^2-(n+1)\lan\xi,rz\ran^2}{|x-rz|^{n+3}} r^{n-1} \big(f(z)-f(x)\big)drdz
+\frac1n f(x) \notag\\
&=&  \int_{\S^n}\hat\omega(x,\xi,z)g(x,z)dz +\frac1n f(x) , 
\eeqn
where $g(x,z) = O(|x-z|^\alpha)$ as $f\in C^\alpha(\S^n)$.
We will verify that
\beq\label{oxxz}
\hat\omega(x,\xi,z) = O(|x-z|^{-n}) \ \ \text{for}  \ z \ \text{near} \ x.
\eeq
Assume \eqref{oxxz} for a moment. 
Then the integral in \eqref{T3.1a.3} is convergent
and by Theorem \ref{T3.1}, one sees that \eqref{cr3b} is 
a necessary and sufficient condition for the convexity.
To verify \eqref{oxxz}, we assume that $x=\e_{n+1}$ and $\xi =\e_1$.
Fix a small $\ve>0$ and consider $z_{1,\ve} = (\ve\e_1+\e_{n+1})/\sqrt{1+\ve^2}$.
Then
\beqn\label{T3.1a.5}
\hat\omega(x,\xi,z_{1,\ve})&\simeq&
\int_0^\infty \frac{r^{n-1}}{((1-r)^2+\ve^2)^{\frac{n+1}{2}}} dr
-\ve^2\int_0^\infty \frac{r^{n+1}}{((1-r)^2+\ve^2)^{\frac{n+3}{2}}} dr \notag \\
&\simeq&\int_{-1}^1 \frac{(t+1)^{n-1}}{(t^2+\ve^2)^{\frac{n+1}{2}}} dt
-\ve^2\int_{-1}^1 \frac{(t+1)^{n+1}}{(t^2+\ve^2)^{\frac{n+3}{2}}} dt \notag \\
&\simeq&\ve^{-n}\int_{-1/\ve}^{1/\ve} \frac{1}{(\rho^2+1)^{\frac{n+1}{2}}} d\rho
-\ve^{-n}\int_{-1/\ve}^{1/\ve} \frac{1}{(\rho^2+1)^{\frac{n+3}{2}}} d\rho \notag \\
&\simeq& \ve^{-n}.
\eeqn
If we consider $z_{2,\ve}=(\ve \e_2+\e_{n+1})/\sqrt{1+\ve^2}$,
then similarly, 
\beqn\label{T3.1a.6}
\omega(x,z_{2,\ve},\xi)
&\simeq&
\int_0^\infty\frac{r^{n-1}}{((1-r)^2+\ve^2)^{\frac{n+1}{2}}}dr\simeq\ve^{-n}.
\eeqn
Apparently \eqref{oxxz} follows from \eqref{T3.1a.5} and \eqref{T3.1a.6}.
\end{proof}

Next we give another sufficient condition
for the convexity of solutions to Christoffel's problem.

\begin{theorem}\label{T3.3}
Let $n\ge2$, $f\in C^{0,1}(\S^n)$ be a positive function.
Extend $f$ to $\R^{n+1}$ such that it is homogeneous of degree $-1$. 
Then the solution to \eqref{e1} is convex, provided 
\beq\label{pxf}
\p_\xi  f(x+t\xi)- \p_\xi  f(x-t\xi)\le 0, \ \ \forall \ t>0, \
\xi\in\S^n,x\in\R^{n+1} \ \text{with} \ \lan\xi,x\ran=0.
\eeq
\end{theorem}

%--------------------------------------
\begin{comment}
{\color{blue}
We observe that if 
\beq\label{strong cond}
\lan\xi,y\ran  \p_{\xi} f (y) \le 0, \ \forall \ \xi\in\S^n \ \text{and} \ y\in\R^{n+1}\setminus\{0\},
\eeq
then $f$ satisfies \eqref{pxf}.
In fact, for given $t>0$, $\xi\in\S^n$ and $x\in\R^{n+1}$ with $\lan\xi,x\ran=0$,
let $y = x+t\xi$ and $\hat y = x-t\xi$.
We have by \eqref{strong cond}
\beqs
\begin{split}
&0\ge \lan \xi,y\ran \p_\xi f(y) = t \p_\xi f(x+t\xi), \\
&0\ge \lan \xi,\hat y\ran \p_\xi f(\hat y) = -t \p_\xi f(x-t\xi).
\end{split}
\eeqs
Since $t>0$, we see that \eqref{pxf} follows.
Note that if $f$ satisfies \eqref{strong cond},
then the convexity of the solution to \eqref{e1} is an immediate consequence of Theorem \ref{T1.1}.
}
\end{comment}
%--------------------------------------

\begin{proof}
By Theorem \ref{T1.1}, it suffices to show that
\beq\label{t3.3.1}
J(z,\xi)<0, \ \ \forall \ z,\xi\in\S^n \ \text{with}  \ \xi\perp z,
\eeq
where
\beqs
J(z,\xi) = \int_{\R^{n+1}}\frac{\lan\xi,y\ran f_\xi(y)}{|z-y|^{n+1}}dy.
\eeqs
For any given such $z$ and $\xi$, we assume by a rotation of coordinates that $\xi=\e_1$ and $z=\e_{n+1}$.
Let $u$ be the solution to \eqref{e2}.
Let $\hat u(y) = u(-y_1,y')$ and $\hat  f(y) = f(-y_1,y')$, 
where $y'=(y_2,\cdots,y_{n+1})$.
Then $\hat u$ is the solution corresponding to $\hat f$,
and  $u_{11}(\e_{n+1})=\hat u_{11}(\e_{n+1})$. 
Hence  
$$u_{11}(\e_{n+1}) = \frac 12 \big[ u+\hat u\big]_{11}(\e_{n+1}) . $$
Consequently,
\beq\label{t3.3.2}
J(z,\xi) = \frac12\int_{\R^{n+1}}
\frac{y_1\big(f_{y_1}(y)+\hat  f_{y_1}(y)\big)}{|\e_{n+1}-y|^{n+1}}dy.
\eeq

From \eqref{pxf}  it follows that
\beqs
f_{y_1}(y)+ \hat  f_{y_1}(y) = f_{y_1}(y_1,y')- f_{y_1}(-y_1,y') \le 0, \ \ \text{for} \ y_1>0.
\eeqs
Replacing $y_1$ by $-y_1$, we infer
\beqs
 f_{y_1}(y)+ \hat  f_{y_1}(y) =-\Big( f_{y_1}(-y_1,y')-  f_{y_1}(y_1,y')\Big)
 \ge 0, \ \ \text{for} \ y_1<0.
\eeqs
Combining the above two inequalities, we get
\beq\label{t3.3.3}
y_1\big( f_{y_1}(y)+ \hat  f_{y_1}(y) \big)\le0, \ \ \forall \ y\ne 0.
\eeq
This together with \eqref{t3.3.2} shows that $J(\xi,z)\le0$.

Suppose $J(z,\xi)=0$.
Then \eqref{t3.3.2} and \eqref{t3.3.3} implies
\beqs
\p_{y_1}\big( f(y_1,y')+f(-y_1,y')\big) = f_{y_1}(y)+\hat  f_{y_1}(y)=0
\ \ \text{for almost all} \ y\in\R^{n+1}.
\eeqs
As a consequence, there is a function $h$ defined in $\R^n$ such that, for almost all $y'\in\R^n$,
\beqs
 f(y_1,y')+ f(-y_1,y') = h(y').
\eeqs
Since $ f$ is homogeneous of degree negative one,
we conclude that, by sending $y_1\to\infty$, $h(y')=0$.
Therefore, for almost all $y'\in\R^n$,
\beqs
f(y_1,y')+ f(-y_1,y') =0.
\eeqs
This contradicts our assumption $f>0$.
We complete the proof.
\end{proof}

% % We observe that if 
% % \beq\label{strong cond}
% % y_i  f_{y_i} (y) \le 0, \ \forall \ y\in\R^{n+1}\setminus\{0\}, \ i=1,\ldots,n+1,
% % \eeq
% %  then $f$ satisfies \eqref{pxf}.
% % In fact, by replacing $y_i$ with $-y_i$ in \eqref{strong cond},
% % we have $ f_{y_i}(\hat y)\ge 0$ when $y_i\ge0$, where
% % $\hat y=(y_1,\cdots,y_{i-1},-y_i,y_{i+1},\cdots,y_{n+1})$.
% % Hence if \eqref{strong cond} holds, then
% % \beq\label{T3.3 cond2}
% % f_{y_i}(y)-f_{y_i}(\hat y)\le 0, \ \forall \ y\in\R^{n+1} \ \text{with} \ y_i>0,  \ i=1,\ldots,n+1.
% % \eeq
% % Note that we have the decomposition $y=x+t\e_i$,
% % where $x\cdot\e_i=0$.
% % Thus $\hat y=x-t\e_i$.
% % Therefore conditions \eqref{pxf} and \eqref{T3.3 cond2} are equivalent.

We finish this section by a few remarks. 

 \vskip10pt

\noindent {\bf Remark 3.1.} 
Here we state the conditions of Firey \cite{F67} and Berg \cite{B69},  in comparison with our condition \eqref{cr1}.

In \cite{F67}, 
Firey reduced condition \eqref{cr} to 
\beq\label{Fir}
\int_{S^n} \lan x, y'\ran \Theta(x, y) \lan Df(x), y'\ran dx\le 0\  \ \ \forall\ y, y'\in \S^n, y\perp y' ,
\eeq
where 
$$\Theta(x, y)= (1-\lan x, y\ran^2)^{-n/2} \int_\pi^{\text{arccos} (\lan x, y\ran)} \sin^{n-1} t dt. $$

Berg  proved that  the solution is convex if and only if the function
\beq\label{Ber}
\int_{S^n} g_n(\lan x, y\ran) f(x) d x
\eeq
is convex, where 
$$ g_2(t)=\frac 1\pi (\pi-\cos^{-1} t) (1-t^2)^{1/2} -\frac{t}{2\pi}. $$
$$  g_3(t)=1+t\log(1-t)+\big(\frac 43-\log 2) t, $$
and for $n\ge 2$,
$$ 
g_{n+2}(t) =\frac{n+1}{(n-1)^2} t g'_n(t) +\frac{n+1}{n-1} g_n(t) 
  +\frac {t}{\sqrt \pi} \frac{(n+1)\Gamma((n+2)/2)}{(n+2)\Gamma((n+1)/2)} . $$

Our condition \eqref{cr1} is equivalent to \eqref{Fir} or \eqref{Ber},
as all these conditions are derived from the fundamental solution.
The conditions \eqref{Fir} and \eqref{Ber} are derived from the
fundamental solution in $\S^n$ but ours is from that on $\R^n$.
So our condition \eqref{cr1} looks simpler.
Moreover, from our condition \eqref{cr1}, we can derive a number of simpler sufficient conditions
for the convexity of solutions.
It is not so easy to find sufficient conditions from \eqref{Fir} and \eqref{Ber}.

 \vskip10pt

\noindent {\bf Remark 3.2.}  
Firey's condition \eqref{Fir} is equivalent to our condition \eqref{cr1}.
Here we show how  \eqref{Fir} can be derived from \eqref{cr1}.
Let $s=\lan x,z \ran$ and $\theta=\arccos s$. It is not hard to see that
\beqs
\omega(x,z) &=& -\int_0^\infty\frac{r^{n-1}}{(r^2-2sr+1)^{\frac{n+1}{2}}} dr,
\eeqs
where $\omega(x,z)$ is given in Theorem \ref{T3.1a}.
Denote $\rho = (r-s)/\sqrt{1-s^2}$. Then we have
\beqs
\omega(x,z) &=&-(1-s^2)^{-\frac{n}{2}}\int_{-\frac{s}{\sqrt{1-s^2}}}^\infty\frac{(\sqrt{1-s^2}\rho+s)^{n-1}}
{(\rho^2+1)^{\frac{n+1}{2}}} d\rho.
\eeqs
Taking $\rho = -\cot \varphi$ and using $s=\cos\theta$, we further deduce
\beqs
\omega(x,z) &=&-(1-s^2)^{-\frac{n}{2}}\int_{\theta}^\pi\sin^{n-1}(\varphi-\theta)d\varphi \\
&=&\big(1-\lan x,z\ran^2\big)^{-\frac{n}{2}} \int_{\pi}^{\arccos \lan x,z\ran}\sin^{n-1}tdt.
\eeqs
Hence our condition \eqref{cr3a} is equivalent to Firey's condition \eqref{Fir}.

 \vskip10pt

\noindent {\bf Remark 3.3.}  
Pogorelov  \cite{Pog53} established the convexity of solutions to \eqref{e1} when $n=2$ under the condition
\beq\label{PC}
f (x)-  f_{ss}(x)>0 \ \ \text{on} \ \S^2,
\eeq
where the sub-script $s$ means differentiation with respect to arc length of great circle on $\S^2$.
In \cite{GM03}, Guan and Ma studied the Christoffel-Minkowski problem,
which is to find convex bodies with prescribed area measure of order $k$.
When $k=1$, it is the Christoffel problem discussed in this paper.
By proving a constant rank theorem they found the following
sufficient condition for the convexity of solutions to \eqref{e1}, 
\beq\label{GMc}
\D^2 f^{-1}(x) + f^{-1}(x) I \ge 0 \ \ \text{on} \ \S^n.
\eeq
It is easy to see that if $f$ satisfies \eqref{PC}, then it satisfies \eqref{GMc}, for two dimensions.

Our conditions in Theorem \ref{T1.1} and in Theorems \ref{T3.1}-\ref{T3.3}
do not involve the second derivatives of $f$, and so are different from \eqref{GMc}.
 
\vskip20pt

%%%%%%%%%%%%%%%%%%%%%%%%%%%%%%%%%%%%%%%%%%%%%%%%%%%%%%

\section{The $L_p$ Christoffel problem}

In this section we study an extension of the Christoffel problem,
called the $L_p$ Christoffel problem.
The problem was introduced by Hu, Ma and Shen  \cite{HMS04},
and studied later by Guan and Xia \cite{GX18}.
It can be formulated as finding convex solutions to the equation \eqref{e3}.
%\beq\label{e3}
%\Delta_{\S^n} u + nu = u^{p-1} f  \ \ \text{on} \ \S^n.
%\eeq

In \cite{HMS04, GX18}, a constant rank theorem was established for convex solutions to \eqref{e3},
provided that
$f\in C^{1,1}(\S^n)$, $f>0$,  and
\beq\label{GM cond2}
 \D^2 f^{-\frac{1}{p}} + f^{-\frac{1}{p}} I \ge 0 \ \ \text{on} \ \S^n.
\eeq
As a result, the existence of convex solutions to \eqref{e3} was obtained 
in \cite{HMS04} for the case $p\ge2$; 
and an even convex solution to \eqref{e3} was obtained in \cite{GX18} 
if $f$ is also even and $1<p<2$. The papers \cite{HMS04} and \cite{GX18} also  obtained 
similar results for the $L_p$ version of the Christoffel-Minkowski problem.

Equation \eqref{e3} is a semi-linear elliptic equation.
Semi-linear elliptic equations have been extensively studied in the last four decades.
An example is the prescribed scalar curvature eqjuation on $\S^n$, 
\beq\label{e3a}
-\Delta_{\S^n} u + \frac{n(n-2)}{4} u =\frac{n-2}{4(n-1)} u^{\frac{n+2}{n-2}}f \ \ \text{on} \ \S^n, \ n>2.
\eeq
Equation \eqref{e3a} has been studied by numerous authors in the last a few decades,
see e.g. \cite{BC91,CY87,CLi97,CLn01,ES86,KW75}.
But the $L_p$-Christoffel problem, though the equation is so simple, 
received much less attention.
In this paper we will prove the following result.
 
\begin{theorem}\label{T4.1}
Let $n\ge 2$, and $f\in C^{0,1}(\S^n)$ be a positive function.
Given $p\ge2$, if
\beq\label{T4.1 cond}
 \Big(1+\frac{n(p-1)}{n-1}\frac{ \max_{\S^n}f}{\min_{\S^n}f} \Big)
 \Big[ \exp\Big(\frac{n\pi}{n-1}\frac{ \max_{\S^n}|\D f|}{\min_{\S^n}f}\Big)\Big]^{p-1}\max_{\S^n}|\D f| 
 \le \gamma_{n,1}\min_{\S^n}f,
\eeq
then 
\begin{itemize}

\item[(i)] when $p>2$, there is a unique positive convex solution $u$ to \eqref{e3}.
\item[(ii)] when $p=2$, there is a unique $\lambda>0$ such that
\eqref{e3} with $f$ replaced by $\lambda f$
has a unique (up to dilations) positive convex solution $u$.
\end{itemize}
\end{theorem}

When $p\ge2$, the existence and uniqueness of positive solutions
were proved in \cite{HMS04}.
We only need to show that the solution is convex under the condition \eqref{T4.1 cond}.
Note that \eqref{T4.1 cond} is for the convexity, but not for the existence and uniqueness of solutions.

We first present a lemma on the gradient estimate for the solutions to \eqref{e3}, in the case $p\ge2$.
\begin{lemma}\label{L4.1}
Let $p\ge 2$, $f\in C^{0,1}(\S^n)$  be a positive function,
and $u$ be a positive $C^{2}$ solution to \eqref{e3}.
Then
\beqs
\max_{\S^n} \frac{|\D u |}{u}\le\frac{n}{n-1} \frac{ \max_{\S^n}|\D f| }{\min_{\S^n}f}.
\eeqs
\end{lemma}

\begin{proof}

Let $v=\log u$. Since $u$ is solution to \eqref{e3}, we have
\beq\label{L4.1.1}
\Delta_{\S^n} v+|\D v|^2+n = f e^{(p-2)v} \ \ \text{on} \ \S^n.
\eeq
Let $Q= |\D v|^2$. Suppose that $Q$ attains its maximum at $x_0$.
By a rotation of the coordiantes, we may assume that $|\D v|(x_0) = v_1(x_0)$.
Hence
\beqs
0=\D_1Q (x_0)= 2v_1v_{11}.
\eeqs
Differentiating \eqref{L4.1.1} yields, at $x_0$,
\beq\label{L4.1.2}
\sum_iv_{ii1}=(\D_1 f + (p-2)fv_1)e^{(p-2)v}\ge \D_1 f u^{p-2}.
\eeq
Then, by using the Ricci identity,
\beqn
0\ge \Delta_{\S^n} Q(x_0) &=& 2\sum_{i,j} v_{jii}v_j+2\sum_{i,j}v_{ij}^2 \notag\\
&=& 2\sum_{i,j} (v_{iij} + v_j - v_i\delta_{ij})v_j+2\sum_{i,j}v_{ij}^2  \notag \\
&\ge&2v_1\sum_{i} v_{ii1}  +2(n-1)v_1^2  \notag \\
&\ge&2v_1 \D_1fu^{p-2} +2(n-1)v_1^2,\label{L4.1.3}
\eeqn
where \eqref{L4.1.2} is used in the last inequality.
We also have
\beq\label{L4.1.4}
u^{p-2}(x_0)\le (\max_{\S^n} u)^{p-2} \le n/\min_{\S^n}f,
\eeq
where the second inequality above follows by applying the maximum principle to \eqref{e3}.
Plugging \eqref{L4.1.4} into \eqref{L4.1.3}, we get
\beq\label{L4.1.5}
0\ge -2nv_1 \frac{\max_{\S^n}|\D f|}{\min_{\S^n}f}+2(n-1)v_1^2.
\eeq
Since $v_1(x_0) = \max_{\S^n}\frac{|\D u|}{u}$, we complete the proof by \eqref{L4.1.5}.
\end{proof}

We next finish the proof of Theorem \ref{T4.1}.

\begin{proof}

As aforementioned, the existence and uniqueness of positive solutions in Theorem \ref{T4.1} 
were obtained in \cite{HMS04}. We only need to prove the convexity of the solution $u$.

We extend $u$ in $\R^{n+1}$ such that it is homogeneous of degree $1$,
and extend $f$ in $\R^{n+1}$ such that it is homogeneous of degree $-p$.
Then $fu^{p-1}$, after extension, is a homogeneous function of degree $-1$.
For convenience we still use $u$ and $f$ to denote their extensions.
Then $u$ satisfies equation \eqref{e2}
with RHS replaced by $fu^{p-1}$ when $p>2$, or $\lambda fu^{p-1}$ when $p=2$.

Let us first prove part (i) of Theorem \ref{T4.1}.
Replacing $f$ by $fu^{p-1}$ in \eqref{t3.2}, we obtain,
for any $x,\xi \in\S^n$ with $x\perp\xi$,
\beqn\label{T4.1.1}
\sum u_{ij}(x)\xi_i\xi_j
&=& \frac{1}{n}fu^{p-1}(x)
+u^{p-1}(x)\int_{\R^{n+1}}  F_{\xi\xi}(x,y)|y|^{-1}\big(f(y/|y|)-f(x)\big) dy \notag \\
&&+\int_{\R^{n+1}} F_{\xi\xi}(x,y)|y|^{-1}\big(u^{p-1}(y/|y|)-u^{p-1}(x)\big) f(y/|y|) dy \notag \\
&\ge& \frac{1}{n}m_u^{p-1}\min_{\S^n}f
-M_u^{p-1} \max_{\S^n}|\D f|\int_{\R^{n+1}} | F_{\xi\xi}(x,y)||y|^{-1}\dist_{\S^n}(y/|y|,x)dy \\
&&-(p-1)M_u^{p-2}\max_{\S^n}|\D u| \max_{\S^n}f
\int_{\R^{n+1}} |F_{\xi\xi}(x,y)||y|^{-1} \dist_{\S^n}(y/|y|,x)  \notag dy.
\eeqn
where $M_u = \max_{\S^n}u, m_u = \min_{\S^n} u $, and
\beqs
{\begin{split}
F_{\xi\xi}(x,y)  & =\sum \xi_i\xi_j F_{x_ix_j}(x,y)\\ 
   & =\frac{|x-y|^2-(n+1)\lan\xi,y\ran^2}{\omega_n|x-y|^{n+3}}.
   \end{split}}
\eeqs
Since $\lan\xi,y\ran^2=\lan\xi,y-x\ran^2\le |x-y|^2$, we have
\beq\label{T4.1.2}
{\begin{split}
|F_{\xi\xi}(x,y)|  & \le \frac{1}{\omega_n|x-y|^{n+3}}\max\big\{|x-y|^2,(n+1)\lan\xi,y\ran^2\big\} \\
  & \le \frac{ (n+1)}{\omega_n|x-y|^{n+1}}.
     \end{split}}
\eeq
Plugging \eqref{T4.1.2} into \eqref{T4.1.1},
and noting that strict inequality holds in \eqref{T4.1.2} when $y\ne x+t\xi$, $t\in\R$, we obtain
\beqs
\sum u_{ij}(x)\xi_i\xi_j
>\frac1n m_u^{p-1} \min_{\S^n} f-\frac{1}{n\gamma_{n,1}} \Big(M_u^{p-1}\max_{\S^n}|\D f|
+(p-1)M_u^{p-2}\max_{\S^n}|\D u| \max_{\S^n}f\Big),
\eeqs
where $\gamma_{n,1}$ is given by \eqref{gna} (taking $\alpha=1$).
In view of Lemma \ref{L4.1},
\beqs
\max_{\S^n}|\D u| \le M_u \frac{n}{n-1}\frac{\max_{\S^n}|\D f|}{\min_{\S^n}f}.
\eeqs 
Therefore
\beq\label{T4.1.3}
\sum u_{ij}(x)\xi_i\xi_j > \Big[ \gamma_{n,1}\min_{\S^n}f 
- \Big(\frac{n(p-1)}{n-1} \frac{\max_{\S^n}f}{\min_{\S^n}f}+1\Big)
\Big( \frac{M_u}{m_u}\Big)^{p-1}
 \max_{\S^n}|\D f| \Big]\frac{ m_u^{p-1}}{n\gamma_{n,1}}.
\eeq
Note that, by using Lemma \ref{L4.1} again,
\beqs
  { \begin{split}
\frac{M_u}{m_u}
 & \le \exp\Big(\pi \max_{\S^n}\frac{|\D u|}{u}\Big) \\
 & \le \exp\Big(\frac{n\pi}{n-1}\frac{\max_{\S^n}|\D f|}{\min_{\S^n}f}\Big).
   \end{split}}
   \eeqs
This together with \eqref{T4.1.3} shows that $\sum u_{ij}(x)\xi_i\xi_j>0$,
provided \eqref{T4.1 cond} holds.
Hence the solution is convex.

For part (ii) of Theorem \ref{T4.1}, 
we know that after homogeneous extension
$u$ solves \eqref{e1} with RHS being $\lambda fu^{p-1}$.
Hence, replacing $f$ by $\lambda f$ in the above argument,
we see immediately that the solution is convex under the condition \eqref{T4.1 cond}.
\end{proof}

%%%%%%%%%%%%%%%%%%%%%%%%%%%%%%%%%%%%%%%%%%%%%%%%%%%%%%%

%%%%%%%%%%%%%%%%%%%%%%%%%%%%%%%%%%%%%%%%%%%%%%%%%%%%%%%

\end{document}